\documentclass[psamsfonts]{amsart}

\usepackage{amssymb,amsfonts}
\usepackage{url}
\usepackage[all,arc]{xy}
\usepackage{enumerate}
\usepackage{mathrsfs}
\usepackage{tikz-cd}

\newtheorem{thm}{Theorem}[section]
\newtheorem{cor}[thm]{Corollary}
\newtheorem{prop}[thm]{Proposition}
\newtheorem{lem}[thm]{Lemma}

\theoremstyle{definition}
\newtheorem{defn}[thm]{Definition}

\theoremstyle{remark}
\newtheorem{rem}[thm]{Remark}

\makeatletter
\let\c@equation\c@thm
\makeatother
\numberwithin{equation}{section}

\newcommand{\R}{\mathbb{R}}

\newcommand{\Z}{\mathbb{Z}}
\newcommand{\C}{\mathbb{C}}
\newcommand{\HH}{\mathbb{H}}

\renewcommand\Im{\operatorname{Im}}

\newcommand{\SO}{\text{SO}}
\newcommand{\SL}{\text{SL}}

\newcommand{\aand}{\text{ and }}

\title{A New Modular Characterization of the Hyperbolic Plane}

\author{Mark Greenfield and Lizhen Ji}

\date{}

\begin{document}

\begin{abstract}
We develop a natural and geometric way to realize the hyperbolic plane as the moduli space of marked genus 1 Riemann surfaces. To do so, a metric is defined on the Teichm\"uller space of the torus, inspired by Thurston's Lipschitz metric defined in \cite{thurston} for the case of hyperbolic surfaces. Based on extremal Lipschitz maps, the Teichm\"uller space of the torus with this new metric is shown to be isometric to the hyperbolic plane under the usual identification. This also gives a new way to recover the complex-analytic Teichm\"uller metric via metric geometry on the underlying surfaces. Along the way, we prove a few results about this metric analogous to Thurston's Lipschitz metric in the case of hyperbolic surfaces, and analogous to the Teichm\"uller metric. 
\end{abstract}

\maketitle

\tableofcontents

\section{Introduction}
\label{intro}

\noindent Hyperbolic space arises naturally in many contexts. In particular, there are many ways to obtain structures isometric to the hyperbolic plane $\HH^2$. Here, we will develop a new metric on the Teichm\"uller space of the 2-torus $\mathbb{T}^2$ and show that it results in a metric space isometric to the hyperbolic plane. This is also known to be isometric to the Teichm\"uller metric. 

Recall the Teichm\"uller space of a closed oriented surface $S_g$ of genus $g$, denoted by $\mathcal{T}_g$, is the moduli space of equivalence classes of marked complex structures on the surface. As we will review, by the uniformization theorem each such marked complex structure possesses a canonical Riemannian metric, flat in the case of $g=1$ and hyperbolic in the case of $g\geq2$. By the Gauss-Bonnet theorem, the area of any closed hyperbolic surface $S$ is $-2\pi\chi(S)$. However, since flat metrics can take on any volume, in order to obtain a canonical flat metric for a given complex structure, one must specify some form of normalization in addition to constant curvature. We will discuss these definitions in greater detail in Section \ref{background}. 

Several different kinds of metrics have been defined for $\mathcal{T}_g$, including the classical Teichm\"uller metric $d_{Teich}$, based on extremal quasiconformal distortion between two marked complex structures on the surface $S_g$. This metric is the most natural based on the classical definition of $\mathcal{T}_g$ in terms of complex structures on surfaces. It is a remarkable classical result which we describe in Section \ref{background} that the Teichm\"uller space $\mathcal{T}_1$ with the Teichm\"uller metric is isometric to the hyperbolic plane up to a scalar multiple.

The next most well-known metric on $\mathcal{T}_g$ is the Weil-Petersson metric, introduced by Weil in 1958 \cite{weil}. This metric is based on the Petersson inner product defined on modular forms on the upper half-plane $\HH^2$. We will describe this in slightly more detail in Section \ref{background}. Ahlfors showed in 1961 in \cite{ahlfors1} and \cite{ahlfors2} that the Weil-Petersson metric gives a K\"ahler structure on $\mathcal{T}_g$ with negative Ricci, scalar, and holomorphic sectional curvatures. 

In \S7.3.5 of \cite{imayoshi}, a metric is defined on $\mathcal{T}_1$, the Teichm\"uller space of the torus, which is the appropriate genus 1 analog of the Weil-Petersson metric. There, in order to integrate the modular forms on the torus (for the Petersson inner product), a normalization to tori of area 1 is utilized. The resulting metric on $\mathcal{T}_1$ turns out to be isometric to the hyperbolic metric on the upper half plane, up to a scalar multiple. We will exhibit a very similar result using a metric defined very differently. 

In \cite{thurston}, Thurston defined an asymmetric metric $\lambda$ on $\mathcal{T}_g$ for the Teichm\"uller spaces of hyperbolic surfaces with $g\geq2$ using the extremal Lipschitz constant between canonical Riemannian metrics. Let $[S,f]$ and $[S',f']$ be two marked Riemann surfaces in $\mathcal{T}_g$. Let $(S_g,h)$ and $(S_g,h')$ denote the Riemannian manifolds with the associated flat metrics for $[S,f]$ and $[S',f']$ respectively. Recall that for a homeomorphism $\varphi:(S_g,h)\to(S_g,h')$, the Lipschitz constant $\mathcal{L}(\varphi)$ is defined as:
\begin{equation}
\label{lipschitz}
\mathcal{L}(\varphi) = \sup_{x\neq y}\bigg(\frac{d_{h'}(\varphi(x),\varphi(y))}{d_h(x,y)}\bigg).
\end{equation}
When the supremum is finite, $\varphi$ is a Lipschitz map. Then \emph{Thurston's Lipschitz metric} $\lambda(h,h')$ is given by:
\begin{equation}
\label{thurstonlambda}
\lambda(h,h') = \inf_{\varphi}\log(\mathcal{L}(\varphi)),
\end{equation}
where the infimum is over all Lipschitz maps homotopic to $f'^{-1}\circ f$ (i.e. respecting the markings). 

This metric is defined based on the smallest Lipschitz constant possible when distorting one hyperbolic metric into another. One can easily show the triangle inequality, but showing that $\lambda$ separates points is nontrivial. See \cite[Proposition 2.1]{thurston} for a proof. However, symmetry actually fails for hyperbolic surfaces - a counterexample is given in \S2 of \cite{thurston}. Thurston remarked that while one can consider the symmetrization, given by 
$$
S\lambda(h,h') = \frac{1}{2}\big(\lambda(h,h')+\lambda(h',h)\big)
$$
to obtain a symmetric metric, the asymmetric version is useful because of its direct geometric interpretation as quantifying the maximum stretch in each direction. This metric is natural on Teichm\"uller spaces of hyperbolic surfaces using only the canonical Riemannian metric associated to each marked complex structure. It is similar to the Teichm\"uller metric in that both seek extremal values for what amounts to a global stretch factor. 

Another metric defined by Thurston in \cite{thurston} for Teichm\"uller spaces of hyperbolic surfaces is based on ratios of lengths of curves. Recall that a closed curve $c$ on a surface is said to be essential if it is not homotopic to a point, a puncture, or a boundary component. Let $\mathcal{S}(S_g)$ denote the collection of essential curves on the surface $S_g$, and given a curve $\alpha\in\mathcal{S}(S_g)$, denote by $\ell_h(\alpha)$ the length of the geodesic representive for $\alpha$ with the metric $h$. Denoting by $h$ and $h'$ the two hyperbolic metrics obtained from points of Teichm\"uller space, the metric $\kappa$ on $\mathcal{T}_g$ is defined by:
\begin{equation}
\label{thurstonkappa}
\kappa(h,h') = \log\sup_{\alpha\in\mathcal{S}(S_g)}\frac{\ell_h(\alpha)}{\ell_{h'}(\alpha)}.
\end{equation}

In this paper, we define the appropriate analog of Thurston's Lipschitz metric for the case of $g=1$, and use this to develop a new modular interpretation of $\HH^2$. Along the way we also consider $\kappa$ in this setting and prove that $\kappa=\lambda$, just as in \cite{thurston}. Unique to the case of the torus is the fact that constant curvature does not uniquely determine a metric: one may perform a scaling of a flat torus to change the area. We will choose a normalization to area 1 in order to obtain a canonical Riemannian metric on each marked torus. This is one of the major goals of Section \ref{background}. Without a choice of normalization, one could find conformally equivalent flat metrics of any area.

We can extend the definition of $\lambda$ given above to the case of $\mathcal{T}_1$, thereby obtaining the analog of Thurston's metric for hyperbolic Riemann surfaces. We will prove the following result: 

\begin{thm}
\label{main}
The function $\lambda$ in Equation \ref{thurstonlambda} defines a (symmetric) metric on $\mathcal{T}_1$, the Teichm\"uller space of marked flat tori of volume 1. Further, the metric spaces $(\mathcal{T}_1,\lambda)$, $(\mathcal{T}_1,d_{\text{Teich}})$, and $(\HH^2,h)$ are isometric (up to a scalar multiple), where $h$ is the Poincar\'e metric on the hyperbolic plane and $d_{\text{Teich}}$ is the Teichm\"uller metric. 
\end{thm}
Our main tool is Proposition \ref{extremalaffine}, where we show that the extremal Lipschitz distortion is realized by the unique affine map between two tori. This allows us to directly compute the value of the metric $\lambda$. Comparing to the Teichm\"uller metric, we obtain our desired result. Recall that the extremal quasiconformal map realizing the Teichm\"uller distance is unique by Teichm\"uller's uniqueness result (see Theorem 11.9 of \cite{primer}). Interestingly, this is not the case for extremal Lipschitz maps. We give a construction for an infinite family of distinct maps each with the extremal Lipschitz constant between two tori in Proposition \ref{nonuniqueaffine}.

The Teichm\"uller space is well-understood as the moduli space of marked complex structures. Using a canonical association of Riemannian metrics and the natural definition from Thurston's metric, one obtains precisely the hyperbolic plane. As we have seen, both the Teichm\"uller metric and the genus-1 Weil-Petersson metric defined on $\mathcal{T}_1$ yield spaces isometric to the hyperbolic plane. Both of these metrics rely on the complex structures of the underlying Riemann surfaces. Theorem \ref{main} continues this theme without using quasiconformal dilatation or modular forms. Instead, we obtain our new modular interpretation of the hyperbolic plane in terms of the metrics on the underlying surfaces, in contrast to previous work in this direction.

Recently, Belkhirat, Papadapolous, and Troyanov have defined an asymmetric version of Thurston's metric on $\mathcal{T}_1$ using a different notion of normalization which requires a choice of marking and measuring lengths of curves. Interestingly, the paper \cite{sorvali} by T. Sorvali in 1975 treated many of the same ideas as in \cite{bpt} with the same normalization, before even Thurston's original paper \cite{thurston} was written. In particular, the metric on $\mathcal{T}_1$ of Sorvali is equivalent to a certain symmetrization of the weak metric in \cite{bpt}. We will discuss this further in section \ref{sorvalisection}. As mentioned, our new metric is normalized instead by area. Our resulting metric has many desirable properties over previous versions, and we give more geometric motivation and proofs for our results. 

An interesting generalization to consider is the case of higher-dimensional tori $T^n$. Many of the basic results still apply to this case, despite the fact that we will no longer have some of the nicer properties of surfaces. We may still define the space of marked flat $n$-tori and normalize the canonical Riemannian metric by a choice of volume 1. The Teichm\"uller space in question generalizes from $\HH^2 \cong SO(2)\backslash SL(2,\R)$ to the symmetric space $SO(n)\backslash SL(n,\R)$.

We will review the classical definition of $\mathcal{T}_g$ and the correspondence of $\mathcal{T}_1$ with the hyperbolic upper-half plane $\HH^2$ in Section \ref{background}, and obtain a canonical flat Riemannian metric for each marked complex structure. A discussion of the existence of extremal maps will be given in Section \ref{lipschitzsection}, the main result of which is an explicit form for an extremal Lipschitz map. This will allow us to directly compute the metric $\lambda$. We will also exhibit distinct extremal Lipschitz maps in Proposition \ref{nonuniqueaffine}. In section \ref{lambdametric}, we will give a more complete description of Thurston's Lipschitz metric $\lambda$ and prove Theorem \ref{main}.  In section \ref{kappametric} we re-introduce Thurston's metric $\kappa$ in our context, and prove it is equivalent to $\lambda$. We close with a brief discussion in section \ref{sorvalisection} of a comparison of the results of Belkhirat-Papadopoulos-Troyanov in \cite{bpt} and Sorvali in \cite{sorvali}, both of which served as precursors to this work.

\subsection*{Acknowledgements} The authors wish to thank Richard Canary for several useful discussions. The first author is supported by the National Science Foundation Graduate Research Fellowship Program under Grant No. DGE\#1256260. Any opinions, findings, and conclusions or recommendations expressed in this material are those of the authors and do not necessarily reflect the views of the National Science Foundation.

\section{Teichm\"uller spaces of closed Riemann surfaces}
\label{background}

Here, we will review some background on Teichm\"uller spaces, including the canonical association of equivalence classes of constant-curvature Riemannian metrics, the correspondence between the Teichm\"uller space of the torus and the upper half-plane. We will also touch on the Teichm\"uller metric and the Weil-Petersson metric. The correspondence will give us the tools to prove a few details about the metrics $\lambda$ and $\kappa$, as well as point to some concepts which will generalize to higher dimensions. 

Let $S_g$ be a closed, oriented smooth surface of genus $g\geq1$. 

\begin{defn}
\label{classical}
The Teichm\"uller space $\mathcal{T}_g$ is defined as the set of equivalence classes of marked closed Riemann surfaces of genus $g$:
$$
\mathcal{T}_g = \{[S,f]: S \text{ a Riemann surface}, f:S \to S_g \text{ orientation-perserving homeomorphism}\}/\sim
$$
where the equivalence relation is isotopy. In other words, $[S,f]\sim[S',f']$ if and only if there exists a conformal homeomorphism $h$ such that the following diagram commutes up to homotopy:

\begin{figure}[!ht]
\centering
\begin{tikzcd}[row sep=tiny]
S \arrow[rd, "f"] \arrow[dd, "h"] & \\
 & S_g \\
S' \arrow[ur, "f'"'] &
\end{tikzcd}
\end{figure}
\end{defn}

\begin{rem} 
Notice that by forgetting the maps $f$ and $f'$, we forget the markings of $S$ and $S'$, and the condition reduces to conformal equivalence. The resulting collection defines the classical moduli space $\mathcal{M}_g$ of the surface $S_g$. More formally, one may realize the moduli space as the quotient $$\mathcal{M}_g=\text{Mod}_g\backslash\mathcal{T}_g,$$ where $\text{Mod}_g$ is the mapping class group of $S_g$. 
\end{rem}

We will start by reviewing the correspondence between marked complex structures and equivalence classes of constant-curvature metrics. We will discuss the genus 1 case for the following classical result:

\begin{prop}
\label{canonicalbij}
For each $g\geq1$, there is a canonical bijection
$$
\mathcal{T}_g\cong \text{Met}_g/\text{Diff}_0(S_g)
$$
where $\text{Met}_g$ is the collection of constant-curvature metrics on $S_g$ of area $1$ if $g=1$ or $4\pi(g-1)$ if $g>1$, and $\text{Diff}_0(S_g)$ is the collection of diffeomorphisms of $S_g$ isotopic to the identity. 
\end{prop}

This is a special case of Theorem 1.8 in \cite{imayoshi}. By the uniformization theorem, each complex structure in $\mathcal{T}_g$ is conformally equivalent to a quotient of the complex plane $\C$ with flat metric if $g=1$, or the upper half-plane $\HH^2$ with hyperbolic metric if $g>2$, by the action of $\pi_1(S_g)$ given by deck transformations.

Henceforth we focus on the case $g=1$, where $S_g$ is the torus $\mathbb{T}^2$. By the above, all closed genus 1 Riemann surfaces are given by $\C/\Lambda$, where $\Lambda\subset\C$ is a lattice in $\C$. Now, the marking $f:S\to S_g$ specifies a basis of the lattice by taking the pre-image of a basis for the lattice generating $S_g$ under the lift of $f$. We may take the standard basis $\{1,i\}\subset\C$ for $S_g$. Notice that the operations of rotation or scaling do not alter the marking or conformal class of the quotient. Hence by an appropriate homothety, we may assume that the basis consists of $\{1,\zeta\}$ for some $\zeta\in\C$, and since $f$ is orientation-preserving, we have that $\zeta\in\HH^2$, the upper half-plane. This shows each marked torus may be associated to some $\zeta\in\HH^2$. 

It is not hard to see that such a $\zeta$ is unique for each marked torus. If $\{1,\zeta\}$ and $\{1,\zeta'\}$ are two distinct bases, there are two cases. If the generated lattices are the same, then they differ by the action of an element of $SL(2,\Z)$, which is a change of marking. If the lattices are distinct, then the quotient tori are not conformally equivalent. In either case, they correspond to different elements of Teichm\"uller space. This shows the correspondence between the Teichm\"uller space of the torus and the hyperbolic plane $\HH^2$ is bijective.  

We will now see how to obtain the canonical class of Riemannian metrics for each marked torus. Pick $[S,f]\in\mathcal{T}_1$, and let $\Lambda$ be an associated $\C$-lattice as obtained above with a choice of basis $\{1,\zeta\}$ corresponding to the marking. The Euclidean metric on $\C$ descends to a flat metric on the quotient $\C/\Lambda$. Consider a fundamental parallelogram $P\subset\C$ spanned by this basis. Without changing the conformal structure or marking, we may scale the parallelogram to have area 1, obtaining the new basis:
$$
\bigg\{\frac{1}{\sqrt{\Im(\zeta)}},\frac{\zeta}{\sqrt{\Im(\zeta)}}\bigg\}
$$
Now, the only operation on $\C$ fixing the corresponding element of Teichm\"uller space which preserves volume and orientation is a rotation. Thus, up to the action of $\SO(2)$, the choice of parallelogram is unique.

Notice further that any area 1 parallelogram may be obtained up to congruence and orientation uniquely (up to the action of $SO(2)$) from the unit square in the first quadrant by an element of $SL(2,\R)$. We have now demonstrated the identification
$$
\mathcal{T}_1 \leftrightarrow \SO(2)\backslash\SL(2,\R)
$$
between the Teichm\"uller space of the torus and the symmetric space $\SO(2)\backslash\SL(2,\R)$. A representative $M$ of an element of this symmetric space may be interpreted as a flat metric on the torus induced via the quotient of the parallelogram $P$ based at the origin, where $P$ has sides $[M(1),M(i)]$. This is a canonical association of an element of the Teichm\"uller space of the torus to a class of flat metrics. 

Now, consider the transitive and faithful action of $\SL(2,\R)$ on $\HH^2$ by fractional linear transformations on complex numbers. The stabilizer of $i\in\HH^2$ is given by the group $\SO(2)\subset\SL(2,\R)$. Using standard results from the theory of Lie groups, one obtains the well-known isometric identification of $\HH^2$ with the symmetric space
$$
\SO(2)\backslash\SL(2,\R)\leftrightarrow\HH^2.
$$
One can see that the action of $\SL(2,\R)$ on $\SO(2)\backslash\SL(2,\R)$ defined via matrix multiplication is equivariant with respect to the identification with $\HH^2$. This completes the identifications of the Teichm\"uller space of the torus $\mathcal{T}_1$, the symmetric space $\SO(2)\backslash\SL(2,\R)$, and the hyperbolic plane $\HH^2$.

\begin{rem} 
The importance of the marking is more evident from the fact that it encodes the topological data of a choice of basis for the first homology group. Given a choice of basis of $H_1(S_g)$ on the fixed surface $S_g$, one may directly compare elements of the homology groups on different marked surfaces by pulling back via the marking. More concretely, this results in a choice of fundamental parallelogram in $\C$ whose sides descend to the generating curves. One further obtains a canonical choice of homotopy class of maps between representative surfaces for elements of Teichm\"uller space. For instance, in the case of $g=1$ one may choose the fixed surface to be 
$$
S_1=\C/(\Z+i\Z)
$$
with the projections of the line segments starting at the origin and ending at 1 and $i$ as a concrete choice of generating set for $H_1(S_1)$. Consider then another Riemann surface, say 
$$
S = \C/(\Z+\tau\Z)
$$
for some $\tau$ in the upper half-plane, with marking $f:S\to S_1$. Notice that the choice of basis for the lattice $\Lambda = \Z+\tau\Z$ is not unique for the conformal class of $S$, and the action of $SL(2,\Z)$ on $\Lambda$ will give infinitely many different possibilities. Fortunately, the induced map on homology makes a choice for us:
$$
f_*:H_1(S)\to H_1(S_1).
$$
Because $f$ is an orientation-preserving homeomorphism, the preimages $f_*^{-1}([1])$ and $f_*^{-1}([i])$ constitute a basis of $H_1(S)$. Next, with an appropriate homothety we may require the lift $\tilde{f}:\C\to\C$ to fix 0. Then $\tilde{f}^{-1}(1)$ and $\tilde{f}^{-1}(i)$ are the endpoints of line segments starting at the origin which project down to curves on $S$ which generate $H_1(S)$. We can even go one step further and note that these endpoints will be the images of $1$ and $\tau$ under some element of $SL(2,\Z)$ corresponding to the basis chosen by the marking. 
\end{rem}

Next, we will review the Teichm\"uller metric, which is the first metric to be defined on $\mathcal{T}_g$, and is a natural definition using the complex-analytic definition of Teichm\"uller space. Note that the following definition applies to the Teichm\"uller space $\mathcal{T}_g$ for any $g\geq1$. Let $[S,f],[S',f']\in\mathcal{T}_g$. Then the map $f'^{-1}\circ f$ is an orientation-preserving homeomorphism from $S$ to $S'$. For a homeomorphism $g:S\to S'$, let $K_g$ denote the supremum of the dilatation of the map $g$ across all points where $g$ is differentiable. Recall that the quasiconformal distortion of a (real) differentiable map $g:\C\to\C$ is given by:
\begin{equation}
\label{qcdef}
K_g = \sup_{\C}\frac{|g_z|+|g_{\bar{z}}|}{|g_z|-|g_{\bar{z}}|}.
\end{equation}
Then the \textit{Teichm\"uller metric} on $\mathcal{T}_g$ is defined as:
$$
d_{Teich}([S,f],[S',f']) = \frac{1}{2}\log \inf_{g\in[f'^{-1}\circ f]}(K_g)
$$
where the infimum is taken over all homeomorphisms which are smooth except at finitely many points $g$ in the homotopy class $[f'^{-1}\circ f]$. One can show this defines a metric on $\mathcal{T}_g$; the proof is relatively straightforward using basic properties of quasiconformal distortion and is given in \S5.1 of \cite{imayoshi}. 

We now restrict to the case of genus $g=1$. First, recall the Poincar\'e metric on $\HH^2$, given by:
$$
d_{\HH^2}(z_1,z_2) = \log\frac{|z_1-\bar{z_2}| + |z_1-z_2|}{|z_1-\bar{z_2}|-|z_1-z_2|}
$$
This is the Riemannian distance of the hyperbolic metric on the upper half-plane $\HH^2$, which is given by
$$
ds^2 = \frac{dx^2+dy^2}{y^2}.
$$

Recall the following remarkable theorem of Teichm\"uller in \cite{teich}:
\begin{prop}
\label{h2t1equiv}
Under the identification described above of the hyperbolic plane $\HH^2$ with $\mathcal{T}_1$ equipped with the metric $d_{\text{Teich}}$, the two are isometric metric spaces up to a multiplicative factor of 2. 
\end{prop}
This fact is less surprising when one considers the intimate relationship between quasiconformal distortion and the Poincar\'e metric. We give a very rough description inspired by \S6.5 and \S6.6 of \cite{lehto}. Note that this section only depends on the (marked) complex structure, so no normalization is required.

Let $$S=\C/(\Z+\tau\Z)\ \text{and}\ S' = \C/(\Z+\tau'\Z)$$ be two tori with distinguished bases 
$$\{1,\tau\} \aand \{1,\tau'\}.$$ 
The appropriate homotopy class for maps $\psi:S\to S'$ is determined by the following property: the lifts to $\C$ must satisfy $$0\mapsto0,\ 1\mapsto1, \aand \tau\mapsto\tau'.$$ 

By a theorem of Teichm\"uller in \cite{teich} (see also \cite[Theorem V.6.3]{lehto}), the minimal quasiconformal distortion among maps $\psi:S\to S'$ in the appropriate homotopy class is realized by the projection of the unique affine map $\tilde{\psi}:\C\to\C$ determined by these conditions. An affine map has the same quasiconformal distortion at any point, so the distortion is ``evenly distributed" across the domain. Because it is affine, one can write this map $\tilde{\psi}:\C\to\C$ in the form 
$$
\psi(z) = \bigg(\frac{\tau'-\bar{\tau}}{\tau-\bar{\tau}} \bigg) z+\bigg(\frac{\tau'-\tau}{\tau-\bar{\tau}}\bigg)\bar{z}.
$$ 
This map has constant quasiconformal distortion across the domain, which is relatively straightforward to compute. For the map $\psi$ defined above, we have, by (\ref{qcdef}):
$$
\log K_{\psi} = \log\frac{|\tau'-\bar{\tau}|+|\tau'-\tau|}{|\tau'-\bar{\tau}|-|\tau'-\tau|} = h(\tau,\tau')
$$
where $h$ is the Poincar\'e metric. 

Next, we will very briefly discuss the Weil-Petersson metric, first explored in \cite{weil}. See Chapter 7 of \cite{imayoshi} for a full presentation. We will start with the usual case of hyperbolic surfaces with $g\geq2$. Let $[S,f]$ represent an element $\eta\in\mathcal{T}_g$. Let $Q(S)$ be the vector space of holomorphic quadratic differentials on $S$. These may be written as expressions of the form $\varphi(z)dz^2$, where $\varphi$ is a holomorphic function on $S$ and $z$ is a complex coordinate for $S$; this notation emphasizes that they transform like 2-forms. More abstractly, they are holomorphic sections of the symmetric square of the holomorphic cotangent bundle of $S$. There is a canonical bijective identification of $Q(S)$ with the cotangent space $T^*_{\eta}\mathcal{T}_g$. Now, $Q(S)$ may be equipped with an inner product. For $q_1,q_2\in Q(S)$ this is defined as:
$$
\langle q_1,q_2\rangle = \int_S\bar{q_1}q_2(ds^2)^{-1},
$$
where $ds^2$ is the hyperbolic metric on the Riemann surface. This is a Hermitian cometric, which induces an inner product on the tangent space $T_{\eta}\mathcal{T}_1$s. The resulting metric is known as the Weil-Petersson metric. Ahlfors showed this gives a K\"ahler structure on $\mathcal{T}_g$ with negative Ricci, scalar, and holomorphic sectional curvatures (\cite{ahlfors1}, \cite{ahlfors2}). 

In genus $g=1$, Imayoshi and Taniguchi in \S7.3.5 of \cite{imayoshi} have given an analogous definition of Weil-Petersson metric on $\mathcal{T}_1$ using the identification of $\HH^2$ with $\mathcal{T}_1$. If $\tau\in\HH^2$ represents an element $[S,f]$ of Teichm\"uller space, the genus 1 analog of the norm given by the Petersson inner product is defined by:
$$
\langle \frac{\partial}{\partial\tau},\frac{\partial}{\partial\tau}\rangle = \iint_{S}|\tau-\bar{\tau}|^{-2}\lambda_{\tau}^2dxdy = \frac{1}{4\Im(\tau)^2}
$$
where $\lambda_{\tau} = 1/\sqrt{\Im(\tau)}$ normalizes the area of $S$ to 1. Notice that the tangent and cotangent spaces of $\mathcal{T}_1$ are complex dimension 1, so we need only define the inner product for one generator. The first term $|\tau-\bar{\tau}|^{-2}$ comes from the derivative of the Beltrami coefficient for a map representing an infinitesimal change, $S_{\tau}\to S_{\tau+t}$ for small $t$, in Teichm\"uller space. This is the appropriate analog for the Teichm\"uller space of genus 1 Riemann surfaces. Putting it all together, one obtains a metric on $\mathcal{T}_1$:
$$
ds_{WP}^2 = \frac{1}{2\Im(\tau)^2}|d\tau|^2
$$
which matches the hyperbolic metric on $\HH^2$ up to a multiple of 2.

We summarize the above discussion in the following theorem, referring the reader \cite{lehto}, Theorem 6.4, for details:
\begin{thm}
The map
$$
j:\HH^2\to\mathcal{T}_1
$$
which associates to each element $\zeta\in\HH^2$ the equivalence class $[S_{\zeta},f_{\zeta}]$ of the flat marked torus $S_{\zeta}$ is an isometry if $\HH^2$ is equipped with the Poincar\'e metric and $\mathcal{T}_1$ with the Teichm\"uller metric, and also if $\mathcal{T}_1$ is equipped with the genus 1 Weil-Petersson metric, up to a factor of 2. 
\end{thm}

We will end with a brief but important lemma which follows almost immediately from the work already done in this section. 

\begin{lem}
\label{funddomain}
Given an element $[S,f]\in\mathcal{T}_1$, the space of marked flat tori of area 1, there exists a unique parallelogram $P\subset\C$ with the following properties:
\begin{itemize}
\item $P$ has area 1
\item $P$ is based at the origin 
\item One side of $P$ lies along the positive real axis 
\item The interior of $P$ is contained in the upper half plane
\item The spanning sides based at the origin descend to the generators of $H_1(S)$ determined by the marking $f$
\end{itemize}
\end{lem}
\begin{proof}
Let $[S,f]\in\mathcal{T}_1$. By the uniformization theorem and the standard correspondence $\mathcal{T}_1\leftrightarrow\HH^2$, we obtain a unique $\tau\in\HH^2$ such that $S$ is the quotient of the $\C$-lattice $\Lambda$ generated by $\{1,\tau\}$ with the generators corresponding to the marking as prescribed by condition 5. Define a new lattice $\Lambda'$ which is generated by $\{1/\sqrt{\Im\tau}, \tau/\sqrt{\Im\tau}\}$. Then the parallelogram $P$ is described by the following vertices: 
$$
\bigg\{0,\ \frac{1}{\sqrt{\Im\tau}},\ \frac{\tau}{\sqrt{\Im\tau}},\ \frac{1}{\sqrt{\Im\tau}}+\frac{\tau}{\sqrt{\Im\tau}}\bigg\}.
$$
The sides based at the origin clearly still correspond to the same marking $f$. Uniqueness follows from uniqueness of $\tau$, and the other properties are obvious. 
\end{proof}

The parallelogram $P$ described above is also a \textit{fundamental domain} for the action of the appropriate lattice $\Lambda\subset\C$ which gives $\C/\Lambda = S$.

\section{Extremal Lipschitz maps between tori}
\label{lipschitzsection}

Let $[S,f]$ and $[S',f']$ be two elements of the Teichm\"uller space $\mathcal{T}_1$, with $S=\C/\Lambda$ and $S'=\C/\Lambda'$. We will show that the map $\psi:S\to S'$ which lifts to an affine map $\tilde{\psi}:\C\to\C$ compatible with the marking realizes the minimal Lipschitz constant in its homotopy class. Recall that the chosen homotopy class is given by $[f'^{-1}\circ f]$. As pointed out in \S2 earlier, the map $\psi$ also realizes the Teichm\"uller distance. 

Recall the following useful results \cite[Lemma V.6.2, Theorem IV.3.5]{lehto}:

\begin{prop}
\label{lehtohomotopy1}
\item The group of conformal self-maps on a torus acts transitively on the torus.
\end{prop}

\begin{prop}
\label{lehtohomotopy2}
Two homeomorphisms $g_i:S\to S'$ between Riemann surfaces induce the same isomorphism between the group of deck transformations acting on the universal cover if they are homotopic.
\end{prop}
We will use the above to prove the following result.

\begin{prop}
\label{extremalaffine}
The map $\psi:S\to S'$ which lifts to the unique affine map $\tilde{\psi}:\C\to\C$ described above realizes the minimal Lipschitz constant in its homotopy class.
\end{prop}
\begin{proof}
Let $S=\C/\Lambda$ and $S'=\C/\Lambda'$ be tori of volume 1 with markings $f$ and $f'$. By Proposition \ref{lehtohomotopy1}, and the fact that conformal self-maps of $\C$ are affine, we may consider only those lifts of maps $\varphi:S\to S'$ which have the property that $\tilde{\varphi}(0)=0$. Let $\mathcal{F}$ denote the class of all such lifts whose quotients are homotopic to $f'^{-1}\circ f$. For $g\in \mathcal{F}$, let $\bar{g}$ denote the induced map $S\to S'$. 

Let $q:\C\to\C/\Lambda$ and $q':\C\to\C/\Lambda'$ be the quotient maps. Then for all $g\in \mathcal{F}$, the following diagram commutes:

\begin{figure}[!ht]
\centering
\begin{tikzcd}
\mathbb{C} \arrow[r, "g"] \arrow[d, "q"] & \mathbb{C} \arrow[d, "q'"] \\
S \arrow[r, "\bar{g}"] & S' 
\end{tikzcd}
\end{figure}

 Let $\{\omega_1,\omega_2\}$ be a basis of $\Lambda$, and let $P\subset\C$ be the parallelogram spanned by this basis; it is a fundamental domain for $\Lambda$. For any $g_1,g_2\in \mathcal{F}$, it follows that $g_1(\omega_i)=g_2(\omega_i)+\lambda_i$ for some $\lambda_i\in\Lambda$ for each of $i=1,2$ by commutativity. By Proposition \ref{lehtohomotopy2}, it follows that $\lambda=0$ since $g_1$ and $g_2$ are homotopic. One then obtains a basis $\{\zeta_1,\zeta_2\}$ of $\Lambda'$ such that $\mathcal{F}$ is the class of homeomorphisms $g:\C\to\C$ with
\begin{equation}
\label{gclass}
g(0)=0,\ g(z+m\omega_1+n\omega_2) = g(z) + m\zeta_1+n\zeta_2
\end{equation}
for all $z\in\C$. Notice that any homeomorphism $\C\to\C$ satisfying (\ref{gclass}) descends to a map $S\to S'$ homotopic to $f'^{-1}\circ f$, since the markings are respected. The condition of being affine uniquely determines such a map inside a fundamental domain of $\Lambda$, and hence on all of $\C$. This proves uniqueness of the affine map; let $w\in \mathcal{F}$ be this map.

Next, let $g\in \mathcal{F}$ be a $K$-Lipschitz map, i.e.
\begin{equation}
\label{lipschitzconstant}
K \geq \sup_{x\neq y}\frac{|g(x)-g(y)|}{|x-y|}.
\end{equation}
Define $g_k(z) = g(kz)/k$ for $k=1,2,\ldots$. Then for each $k$, it is not hard to check using (\ref{lipschitzconstant}) that $g_k$ is also $K$-Lipschitz. Further, notice that every $g_k$ satisfies (\ref{gclass}), and so $g_k\in F$ for all $k>0$. We claim that $g_k\xrightarrow{k\to\infty}w$ uniformly on $\C$. We give a proof of this claim below in Lemma \ref{uniformconv}. It is a standard fact from real analysis that the pointwise (and hence, uniform) limit of a sequence of $K$-Lipschitz functions is also $K$-Lipschitz. Hence $w$ is $K$-Lipschitz. In other words, $K\geq\mathcal{L}(w)$. Because this holds for any Lipschitz function in $\mathcal{F}$, we see that $w$ is extremal (i.e. has minimal Lipschitz constant). 
\end{proof}

The following is used without proof in \cite{lehto} for a similar proposition, but a proof is given here. 

\begin{lem}
\label{uniformconv}
In the proof of Proposition \ref{extremalaffine}, the sequence $g_k\to w$ uniformly.
\end{lem}
\begin{proof}
Pick $\epsilon>0$ and let $z\in\C$. Since $\omega_1$ and $\omega_2$ are $\R$-linearly independent, $z$ may be written as $r\omega_1+s\omega_2$ for some $r,s\in\R$. Let 
$$
M=\sup_{(a,b)\in[0,1]^2}|g(a\omega_1 + b\omega_2)| + |\zeta_1|+|\zeta_2|,
$$
which is finite since $g$ is continuous and this domain is compact. Then for any $z=r\omega_1+s\omega_2\in\C$ and any integer $k>M/\epsilon$, we have:
\begin{equation}
\label{uniformconvergence}
 \big|g_k(r\omega_1+s\omega_2)-w(r\omega_1+s\omega_2)\big| = \frac{1}{k}\big|g(kr\omega_1+ks\omega_2) - (kr\zeta_1+ks\zeta_2)\big|
\end{equation}
since $w$ is the affine map. Write $kr=m_1+t_1$ and $ks=m_2+t_2$ where $t_i\in[0,1)$ and $m_i\in \Z$, for $i=1,2$. Then \ref{uniformconvergence} simplifies to:
$$
\frac{1}{k}|g(t_1\omega_1 + t_2\omega_2)- t_1\zeta_1 - t_2\zeta_2|\leq \frac{1}{k}\big(|g(t_1\omega_1+t_2\omega_2)|+t_1|\zeta_1|+t_2|\zeta_2|\big)\leq\frac{1}{k}M<\epsilon,
$$
because the integer part $m$ of $kr$ factors through $g$. On this domain, we have uniform convergence of $g_k$ to the affine map $w$. Since $\{\omega_1,\omega_2\}$ is a basis of $\C$, the calculation for the rest of $\C$ is similar. 
\end{proof}

It is also known that the extremal map for the Teichm\"uller distance is unique (see \cite{lehto}, Theorem 6.3). Interestingly, in contrast to the quasiconformal case, there are many extremal Lipschitz maps. The main idea in the construction is to keep the map linear in the direction of maximal stretch, but allow for variable pinching in the other directions. The construction and proof are elementary. For the square torus to a rectangular torus, we give the following construction of a 2-parameter family of extremal Lipschitz maps. In contrast to the case of the affine map, the inverses of the maps constructed in Proposition \ref{nonuniqueaffine} are not Lipschitz-extremal.

\begin{prop}
\label{nonuniqueaffine}
There exists a pair of marked flat tori with infinitely many distinct homeomorphisms respecting the markings which all realize the extremal Lipschitz constant.  
\end{prop}
\begin{proof}
We will give one possible construction of an infinite family; it will be evident that there are many more possibilities. Fix $r>1$. Choose $\epsilon\in(-1/2,1/2)$ and $\delta$ such that
$$
\max(0,\frac{1}{r}-\frac{r}{2}+\epsilon r)<\delta < \min(\frac{1}{r},\frac{r}{2}+\epsilon r).
$$
For any $r>1$, there are infinitely many choices of $\epsilon$ and $\delta$ which satisfy these. Let $S$ be the square $[0,1]\times[0,1]\subset\R^2$ and $T$ be the rectangle $[0,r]\times[0,1/r]$. These regions $S$ and $T$ represent fundamental domains for two flat tori. By our earlier work, we know the extremal Lipschitz map is given by $(x,y)\mapsto(rx,y/r)$ with Lipschitz constant $r$. We will find a different homeomorphism with the same Lipschitz constant. Define the map $F:S\to T$ by:
$$
F(x,y) =
\begin{cases}
\big(rx, \frac{1/r-\delta}{1/2-\epsilon} y\big) & y\leq 1/2 - \epsilon\\
\big(rx, \big(\frac{1}{r}-\delta\big) + \frac{y-(1/2-\epsilon)}{1/2+\epsilon}\delta & y\geq 1/2 - \epsilon
\end{cases}
$$
See the figure for an explanation of these values. 
\begin{figure}[htbp]
\label{nonuniquepic}
\includegraphics[width = 0.7\textwidth]{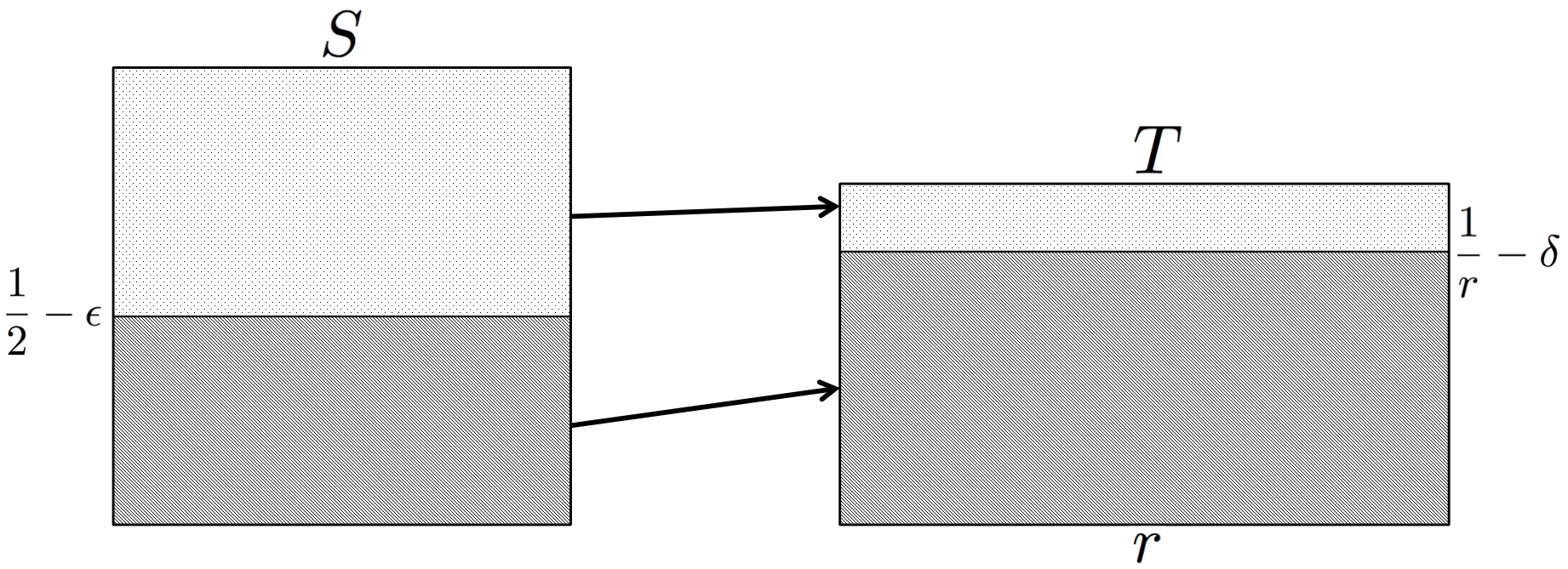}
\caption{The map $F$ sends the two portions of the square linearly to the two portions of the rectangle.}
\end{figure}

This map is linear in the $x$-direction (the direction of maximum stretch), but in the $y$-direction, it pinches less on the bottom half than on the top half (when $\epsilon>0$). The affine map occurs at $\epsilon=0$ and $\delta=1/(2r)$. As $\delta$ decreases, $F$ sends the bottom half of the square $S$ onto a larger portion of the rectangle $T$. As $\epsilon$ increases, a larger portion of the square $S$ would be mapped to the top portion of the rectangle $T$. It is clear that this map projects onto the corresponding tori since it respects the boundaries. This map is differentiable almost everywhere. The total derivatives on the top and bottom halves of the domain are respectively given by:
$$
D_{\text{bottom}} = 
\begin{pmatrix}
r & 0\\
0 & \frac{1/r-\delta}{1/2-\epsilon}
\end{pmatrix},\ 
D_{\text{top}} = 
\begin{pmatrix}
r & 0 \\
0 & \frac{\delta}{1/2+\epsilon}
\end{pmatrix}
$$
For diagonal matrices, the Lipschitz constant is simply the largest entry. With the above constraints on $\epsilon$ and $\delta$, one can check that in both $D_{\text{top}}$ and $D_{\text{bottom}}$ the larger entry is $r$, hence the Lipschitz constant for $F$ is $r$ (see the proof of Theorem \ref{main} for more on finding the Lipschitz constant of a linear map). This is the Lipschitz constant for the extremal map, so $F$ is also extremal. Varying $\epsilon$ and $\delta$, we obtain infinitely many distinct $r$-Lipschitz maps between these two tori, all in the same homotopy class. 
\end{proof}

\begin{rem}
It is straightforward to generalize the above construction to go between any two rectangular tori. Oblique tori are expected to behave similarly. However, to ensure the boundaries are treated properly, the construction will have to be different. For example, one may perturb the affine map only on a small neighborhood of a point in the interior of the parallelogram, with the change in stretch occurring only perpendicular to the direction of maximum stretch. 
\end{rem}

\begin{rem}
The maps constructed in Proposition \ref{nonuniqueaffine} have a larger quasiconformal distortion than that of the affine map (this must be the case by Teichm\"uller's uniqueness result). To see why more concretely, consider for example the case where $\epsilon>0$ and $\delta<1/(2r)$. Notice that in the top half of the domain $S$, a small circle will be stretched into an ellipse with a smaller semiminor (vertical) axis than that of the affine map but the same semimajor (horizontal) axis, giving a larger eccentricity.
\end{rem}

\section{Thurston's Lipschitz metric $\lambda$ on the Teichm\"uller space $\mathcal{T}_1$}
\label{lambdametric}

We will now define the metric $\lambda$ for the case of $\mathcal{T}_1$. In contrast to the case of hyperbolic surfaces, our definition will yield a symmetric metric. Throughout this section, we will denote by $[S,f]$ and $[S',f']$ two marked flat tori in $\mathcal{T}_1$, with $h$ and $h'$ the corresponding volume-1 flat Riemannian metrics. 

Define $\lambda$ as follows:
$$
\lambda([S,f],[S',f']) = \inf_{\varphi}(\log\mathcal{L}(\varphi))
$$
where the infimum is taken over the set of all Lipschitz maps $\varphi:S\to S'$ isotopic to $f'\circ f^{-1}$, where
$$
\mathcal{L}(\varphi) = \sup_{x\neq y}\bigg(\frac{d_{h'}(\varphi(x),\varphi(y))}{d_{h}(x,y)}\bigg).
$$

Next, using Proposition \ref{extremalaffine}, we can directly compute $\lambda$. First, by Lemma \ref{funddomain}, let $P,P'\subset\C$ be fundamental parallelograms of area 1 for $[S,f]$ and $[S',f']$ respectively, based at the origin with one side on the positive real axis. In Proposition \ref{extremalaffine}, we show that the quotient of the affine map $\C\to\C$ which sends $P\to P'$ realizes the extremal Lipschitz constant. Here we can give an explicit form for this map. Notice that between $P$ and $P'$ one may write the map as:
\begin{equation}
\label{standardform}
\varphi: x+iy\mapsto \frac{1}{d}x+cy + idy
\end{equation}
for appropriate choices of $c,d\in\R$, where $d\neq0$. We will use some basic facts from analysis to compute the Lipschitz constant for this linear map.

\begin{lem}
\label{lambdaformula}
The Lipschitz constant $\mathcal{L}(\varphi)$ of the map $\varphi$ in Equation \ref{standardform} is given by:
$$
\mathcal{L}(\varphi) = \bigg( \frac{1}{2} \big(d^2 + d^{-2} + c^2 + \sqrt{(d^2+d^{-2}+c^2)^2-4}\big)\bigg)^{\frac{1}{2}}
$$
\end{lem}

\begin{proof}
One can write the map $\varphi$ as a linear map $M:\R^2\to\R^2$:
$$
M = 
\begin{pmatrix}
\frac{1}{d} & c  \\
0 & d \\
\end{pmatrix}
$$

We will recall some standard facts from linear algebra about the matrix norm. The Lipschitz constant for $\varphi:\C\to\C$ is equal to the operator norm for the matrix $M:\R^2\to\R^2$, using the Euclidean metric:
$$
\mathcal{L}(M) := \sup_{x\neq y}\frac{||Mx - My||}{||x-y||} = \sup_z\frac{||Mz||}{||z||} =: ||M||_{op}
$$
and further recall that for diagonalizable (including symmetric) matrices, the operator norm is simply the magnitude of the largest eigenvalue. Since $MM^T$ is symmetric and $||MM^T||_{op} = ||M||_{op}^2$, the operator norm for $M$ may be computed as follows:
$$
||M||_{op}^2 = ||MM^T||_{op} = \text{max eigenvalue of } 
\begin{pmatrix}
d^{-2}+c^2 & cd \\
cd & d^2 \\
\end{pmatrix}
$$
This may be computed by elementary means, and is given by:
$$
||M||_{op}^2 = \frac{1}{2} \bigg(d^2 + d^{-2} + c^2 + \sqrt{(d^2+d^{-2}+c^2)^2-4}\bigg)
$$
from which our claim is immediate. 
\end{proof}

While our final proof that $\lambda$ is equivalent to the hyperbolic metric will not depend on this fact, we present the following argument of Thurston (adapted from the case for hyperbolic surfaces), which gives a geometric proof of the fact that $\lambda$ is positive definite on $\mathcal{T}_1$. The proof is essentially the same as that of Proposition 2.1 in \cite{thurston}. 

\begin{prop}
\label{thurstonprop}
For all pairs of marked surfaces $[S,f],[S'f']\in\mathcal{T}_1$, we have 
$$
{\lambda}([S,f],[S',f'])\geq0,
$$
with equality only if $[S,f]=[S',f']$, i.e. that $\lambda$ separates points. 
\end{prop}
\begin{proof}
Suppose we have $[S,f],[S',f']$ such that ${\lambda}([S,f],[S',f'])\leq0$. Then by compactness there exists a homeomorphism $\varphi:S\to S'$ in the appropriate homotopy class with global Lipschitz constant $L\leq1$. 

Hence under $\varphi$ every sufficiently small disk of radius $r$ in the domain space is mapped to a subset of a disk of radius $\leq r$ in the range surface. However, both surfaces have the same area. If we cover the domain space by a disjoint union of disks of full measure, one sees that each disk must map surjectively onto a disk of the same size. This procedure works for arbitrarily small disks, and so $\varphi$ is an isometry. 
\end{proof}

From this proposition and a few quick observations, we have the following:

\begin{cor}
The function $\lambda$ is a metric on $\mathcal{T}_1$.
\end{cor}
\begin{proof}
Positive-definiteness is proven in Proposition \ref{thurstonprop}. The triangle inequality is immediate from the fact that composing Lipschitz maps gives a map whose Lipschitz constant is at most the product of the Lipschitz constants of the two maps. Finally, $\lambda$ is symmetric. This can be seen by using the inverse of the matrix $M$ in Lemma \ref{lambdaformula} and noticing that the Lipschitz constant is the same for $M$ and $M^{-1}$. 
\end{proof}

We are now ready to prove equivalence of $\lambda$ with the Teichm\"uller metric.

\begin{thm}
The metric $\lambda$ is equal to the metric $d_{\text{Teich}}$ on $\mathcal{T}_1$, up to a scalar multiple.
\end{thm}
\begin{proof}
Let $[S,f],[S',f']\in\mathcal{T}_1$. Let $\varphi:S\to S'$ be the projection of the affine map in the homotopy class $[f'^{-1}\circ f]$. As we have seen, the map $\varphi$ realizes both the Teichm\"uller distance and the Lipschitz metric. We will compute the quasiconformal distortion of the linear map $\tilde{\varphi}:\C\to\C$, the lift of $\varphi$. 

Since $\tilde{\varphi}$ is linear we need only look at the derivatives at a point. Using the formulas $\tilde{\varphi}_z = \frac{1}{2}(\tilde{\varphi}_x-i\varphi_y)$ and $\tilde{\varphi}_{\bar{z}} = \frac{1}{2}(\tilde{\varphi}_x+i\tilde{\varphi}_y)$ and our description of $\tilde{\varphi}$ as a linear map in equation (\ref{standardform}), we obtain:
$$
|\tilde{\varphi}_z| = \frac{1}{2}\sqrt{(d+d^{-1})^2+c^2},\ |\tilde{\varphi}_{\bar{z}}| = \frac{1}{2}\sqrt{(d-d^{-1})^2+c^2}
$$
simplifying the formula for quasiconformal distortion, we obtain:
$$
K_{\tilde{\varphi}} = \frac{1}{2}\big(d^2+c^2+d^{-2}+\sqrt{(d^2+d^{-2}+c^2)^2-r}\big)
$$
and so we arrive at $K_{\varphi} = \mathcal{L}(\varphi)^2$, since the Lipschitz constant and the quasiconformal distortion will be the same for the map $\varphi$ and its lift $\tilde{\varphi}$. Because $\varphi$ realizes both extremal quasiconformal distortion and extremal Lipschitz constant, we have 
$$
d_{\text{Teich}}([S,f],[S',f']) = 2\lambda([S,f],[S',f']). 
$$
\end{proof}

We quickly arrive at Theorem \ref{main} from here. 

\begin{proof}[Proof of Theorem \ref{main}] 
Since the Teichm\"uller metric matches the hyperbolic metric under the usual identification (Proposition \ref{h2t1equiv}), and by the above theorem the metric $\lambda$ is equal to $d_{\text{Teich}}$ up to a scalar multiple, we arrive at Theorem \ref{main}. Hence, the hyperbolic plane is realized as the moduli space of marked flat tori via the extremal Lipschitz constants between them. 
\end{proof}

\section{The metric $\kappa$ on $\mathcal{T}_1$}
\label{kappametric}

Next, we define another metric, $\kappa$, on $\mathcal{T}_1$. Despite the very different definition, it will turn out to be equal to $\lambda$. Let $\mathcal{S}(\mathbb{T}^2)$ denote the set of isotopy classes of essential closed curves on the 2-torus. For $\alpha\in\mathcal{S}(\mathbb{T}^2)$ and $h$ a metric on $\mathbb{T}^2$, denote by $\ell_h(\alpha)$ the shortest length of any curve in the homotopy class $\alpha$. Recall that for the flat torus, while the curve realizing this length is not unique, the shortest length is well-defined. As above, we will denote by $[S,f]$ and $[S',f']$ two marked flat tori in $\mathcal{T}_1$, with $h$ and $h'$ the corresponding area-1 flat Riemannian metrics on $\mathbb{T}^2$. Now, $\kappa$ is defined as follows: 
\begin{equation}
\label{kappadef}
\kappa([S,f],[S',f']) = \log \sup_{\alpha\in\mathcal{S}(T^2)}\bigg(\frac{l_{h'}(\alpha)}{l_{h}(\alpha)}\bigg)
\end{equation}

That is, $\kappa$ is a measure of the maximum stretch along a geodesic in any homotopy class. Now, we will show that $\kappa=\lambda$. 

\begin{prop}
\label{kappalambdasame}
The metrics $\kappa$ and $\lambda$ are equal on $\mathcal{T}_1$.
\end{prop}
\begin{proof}
It is immediate that $\kappa([S,f],[S',f'])\leq\lambda([S,f],[S',f'])$ for all $[S,f],[S',f']\in\mathcal{T}_1$, since the latter involves a supremum over all geodesic segments rather than only closed geodesics. For the opposite inequality, we will need a more geometric argument. Pass to the universal cover of $S$, and suppose $\varphi:\C\to\C$ is the (lift of the) extremal Lipschitz map between $S$ and $S'$ in the homotopy class $[f'^{-1}\circ f]$. 

Recall that any closed geodesic on $S=\C/\Lambda$ may be represented (up to homotopy) by a line segment starting at $0\in\C$ and ending at a lattice point in $\Lambda\subset\C$. Now, given any point $p\in\C$, the Lipschitz constant of a linear map $\C\to\C$ is realized along some line containing $p$. In particular, for the map $\varphi$, there is a line $L$ containing the origin along which the value of the metric $\lambda$ is realized. If there are two lattice points on $L$, then the segment connecting them is a geodesic whose length is stretched by the same factor as the Lipschitz constant, yielding $\kappa\geq\lambda$, and we are done. 

Suppose now 0 is the only lattice point on $L$. One can find a sequence of lattice points $p_n\in\Lambda$ which approach $L$. Then by continuity, under $\varphi$ the corresponding sequence of closed geodesics will have stretch factors approaching the Lipschitz constant of the map $\varphi$. We conclude $\kappa\geq\lambda$, as required. 
\end{proof}

An alternative but less geometric approach is possible using some interesting calculations done in \cite{bpt}. Proposition \ref{kappalambdasame} follows easily from the analogous result (Theorem 4) in \cite{bpt} using a few additional calculations which convert between the different definitions of Thurston's metrics here and in that paper.

\section{Sorvali's precursor to Thurston's metric}
\label{sorvalisection}

In 1975 in \cite{sorvali}, T. Sorvali defined a metric on $\mathcal{T}_1$ before either the work of Thurston \cite{thurston} in 1986 or Belkhirat-Papadopoulos-Troyanov \cite{bpt} in 2005. We will define this metric, which Sorvali called the dilatation metric, and a related metric from \cite{bpt}. The definitions below are phrased in slightly different language from the original treatments, but they are equivalent to the original definitions and illuminate more clearly the relevant comparisons.

Let $[S,f],[S',f']\in\mathcal{T}_1$ and note that we have $f'^{-1}\circ f:S\to S'$. Here, we use the the usual identification $\mathcal{T}_1$ with $\HH^2$ to obtain a choice of Riemannian metrics $h$, $h'$ on the Riemann surfaces different from the choice we made in earlier sections. The volume is not fixed, but we use lattices of the form $\Z+\zeta\Z$ where one generator is fixed to be $1\in\C$. Equivalently, the choice of normalization here fixes the length of the shortest curve to 1. This is the same normalization used in \cite{bpt}. Define $\delta([S,f],[S',f'])$ as:
$$
\delta([S,f],[S',f']) = \inf\big\{a\geq1: \frac{1}{a} \leq \sup_{\gamma}\frac{\ell_{h'}(f'\circ f^{-1}\circ \gamma)}{\ell_h(\gamma)} \leq a\big\}
$$
where $\gamma$ ranges over all closed curves in $S$. Then the \emph{dilatation metric} $d$ is defined as
$$
d([S,f],[S',f'])=\log\delta([S,f],[S',f']).
$$
One can check that this is a metric on $\mathcal{T}_1$. 

Next, we will define a similar metric from \cite{bpt}, which we will denote by $\kappa'$. Note that in the original work it was denoted by $\kappa$, but we use $\kappa'$ to avoid confusion. Write $g=f'^{-1}\circ f$; this map essentially switches the markings. Now, $\kappa'$ is defined as:
\begin{equation} \label{bptmetric}
\kappa'([S,f],[S',f']) = \sup_{\gamma}\log\bigg(\frac{\ell_{h'}(g\circ\gamma)}{\ell_h(\gamma)}\bigg)
\end{equation}
where the supremum is again taken over all closed paths on $S$. 

We may rewrite $\delta$ as follows:
$$
\delta([S,f],[S',f']) = \max\bigg(\sup_{\gamma} \frac{\ell_{h'}(g\circ \gamma)}{\ell_h(\gamma)} , \sup_{\gamma} \frac{\ell_h(\gamma)}{\ell_{h'}(g\circ\gamma)} \bigg) = \max\bigg(\sup_{\gamma} \frac{\ell_{h'}(g\circ \gamma)}{\ell_h(\gamma)} , \sup_{\gamma} \frac{\ell_h(g^{-1}\circ\gamma)}{\ell_{h'}(\gamma)} \bigg) 
$$
From this the following observation is clear:
$$
d([S,f],[S',f']) = \max\big(\kappa'([S,f],[S',f']),\kappa'([S',f'],[S,f])\big)
$$
and so the dilatation metric is a certain symmetrization of $\kappa'$. 

One of the main results from \cite{bpt} is the following:
\begin{prop}
\label{kappasymm}
The symmetrization
$$
S\kappa'([S,f],[S',f']) = \frac{1}{2}\big(\kappa'([S,f],[S',f'])+\kappa'([S',f'],[S,f])\big)
$$
of the $\kappa'$ metric coincides with the Teichm\"uller metric on $\mathcal{T}_1$, and hence the hyperbolic metric on $\HH^2$ under the standard identification.
\end{prop}
See Theorem 3 in \cite{bpt} for more details. Because $\kappa'$ is not symmetric, the two symmetrizations above are not equivalent, so the dilatation metric is not equivalent to the Teichm\"uller metric. Finally, we recall the following result of Sorvali \cite[\S2, Theorem 2]{sorvali}:
\begin{prop}
The following inequality is sharp for the dilatation metric $d$ and the Teichm\"uller metric $d_{Teich}$:
$$
d([S,f],[S',f'])\leq d_{Teich}([S,f],[S',f'])\leq2d([S,f],[S',f'])
$$
\end{prop}
Thus, in a sense, the $\kappa'$ metric measures the maximum amount one would have to stretch any closed curve when transforming $[S,f]$ to $[S',f']$, while the dilatation metric d measures the maximum required stretch going forwards or backwards. In light of Proposition \ref{kappasymm}, one may interpret the Teichm\"uller metric as measuring the average of the maximum stretch in either direction. The $\kappa'$ metric contains the most information in the sense that given the values of $\kappa'$, one may compute both the dilatation metric $d$ and the Teichm\"uller metric $d_{\text{Teich}}$, but the converse does not hold for either case. As we have seen, matching the hyperbolic metric seems to be a theme for a priori different but ``correct" metrics on $\mathcal{T}_1$. The metrics resulting from the normalization used for $d$ and $\kappa'$, which fixes the shortest curve to length 1, does not yield this desired equivalence.


\begin{thebibliography}{99}

\bibitem{ahlfors1} L. Ahlfors, \textit{Some remarks on Teichm\"uller's space of Riemann surfaces}. Ann. Math. Second Series {\bf 74}, 171--191 (1961). 

\bibitem{ahlfors2} L. Ahlfors, \textit{Curvature properties of Teichm\"uller's space}. J. Anal. Math. {\bf 9}, 161--176 (1961). 

\bibitem{bpt} A. Belkhirat, A. Papadopoulos, M. Troyanov, \textit{Thurston's weak metric on the Teichm\"uller space of the torus}. Trans. Am. Math. Soc. {\bf 357}, No. 8, 3311--3324 (2005).

\bibitem{primer} B. Farb, D. Margalit, \textit{A primer on mapping class groups}. Princeton University Press (2012). 

\bibitem{imayoshi} Y. Imayoshi, M. Taniguchi, \textit{An tntroduction to Teichm\"uller spaces}. Springer-Verlag (1992).

\bibitem{lehto} O. Lehto, \textit{Univalent functions and Teichm\"uller space}. Springer-Verlag MR 0867407 (88f:30073), (1987). 

\bibitem{sorvali} T. Sorvaldi, \textit{On Teichm\"uller spaces of tori}. Ann. Acad. Sci. Fenn. Math. {\bf 1}, 7--11 (1975). 

\bibitem{teich} O. Teichm\"uller, \textit{Extremale quasikonforme Abbildungen und quadratische Differentiale}. Abh. Preuss. Akad. Wiss. Math.-Natur. Kl. Nr. 22, 1--197 (1940). 

\bibitem{thurston} W.P. Thurston, \textit{Minimal stretch maps between hyperbolic surfaces}. Preprint (1985) available at http://arxiv.org/abs/math.GT/9801039.

\bibitem{weil} A. Weil, \textit{Modules des surfaces de Riemann.} S\'eminaire Bourbaki {\bf 4}, 413--419 (1958). 

\bibitem{wolpert} S. Wolpert, \textit{The Weil-Petersson metric geometry}, in A. Papadopoulos, \textit{Handbook of Teichm\"uller theory, Vol II}. IRMA Lect. Math. Theor. Phys., {\bf 13}, Eur. Math. Soc., Z\"urich, 47--64 (2009). 
\end{thebibliography}
\end{document}